\documentclass[reqno]{gen-j-l}
\usepackage{amsmath,amssymb,epsfig,graphicx}
\usepackage{hyperref}
\usepackage[hmarginratio=1:1, bottom=3cm]{geometry}
\usepackage[normalem]{ulem}

\usepackage[utf8]{inputenc} 
\usepackage{color}
\usepackage{enumerate}
\usepackage{bbm}
\usepackage{setspace}
\usepackage{subfig}
\usepackage{subfloat}
\usepackage{enumitem}
\usepackage{enumerate}
\usepackage{cancel}

\newtheorem{theo}{Theorem}[section]

\newtheorem{prop}{Proposition}[section]

\newcommand{\cA}{{\mathcal A}}

\newcommand\R{{\mathbb R}}

\newcommand\E{{\mathbb E}}

\newcommand\N{{\mathbb N}}

\newcommand\vare{{\varepsilon}}
\definecolor{newgreen}{rgb}{0,0.6,0.3}

\newcommand{\h}{h}
\newcommand{\vh}{v_{\h}}
\newcommand{\uh}{u_\h}

\newcommand{\uhe}{u_\h^{(\vare)}}
\newcommand{\Easx}{\mathbb{E}_{s,x}^a}
\newcommand{\Ealsx}{\mathbb{E}_{s,x}^\alpha}

\newcommand{\fa}{f_a}
\newcommand{\La}{L_a}
\newcommand{\fat}{f_{\alpha_t}}

\newcommand{\Lat}{L_{\alpha_t}}

\newcommand{\fae}{f_a}
\newcommand{\Lae}{L_a}

\newcommand{\dt}{\, \mathrm{d}t}
\newcommand{\dr}{\, \mathrm{d}r}

\newcommand{\twoint}{\int_0^\h \int_0^t}

\newcommand{\beqn}{\begin{equation}}
\newcommand{\eeqn}{  \end{equation}}
\newcommand{\beqno}{\begin{equation*}}
\newcommand{\eeqno}{  \end{equation*}}
\newcommand{\be}{\begin{eqnarray}}
\newcommand{\ee}{  \end{eqnarray}}
\newcommand{\beno}{\begin{eqnarray*}}
\newcommand{\eeno}{  \end{eqnarray*}}

\numberwithin{equation}{section}
\usepackage{floatrow}

\begin{document}

\author[Espen R. Jakobsen]{Espen R. Jakobsen}
\address{Department of Mathematical Sciences, Norwegian University of Science and Technology, 7491 Trondheim, N}
\email{erj@math.ntnu.no}
\author[Athena Picarelli]{Athena Picarelli}
\address{Department of Economics, University of Verona, via Cantarane 24, 37129 Verona, I}
\email{athena.picarelli@univr.it}
\author[Christoph Reisinger]{Christoph Reisinger}
\address{Mathematical Institute, University of Oxford, Andrew Wiles Building, OX2 6GG, Oxford, UK}
\email{christoph.reisinger@maths.ox.ac.uk}

\title[Order 1/4 convergence of piecewise policies]{Improved order 1/4
  convergence for piecewise constant policy approximation of
  stochastic control problems}
\maketitle 

\begin{abstract}
In \emph{N.~V.~Krylov, Approximating value functions for controlled degenerate diffusion processes by using piece-wise constant policies,
Electron.\ J.\ Probab., 4(2), 1999}, it is proved under
standard assumptions that the value functions of controlled diffusion processes can be approximated with order 1/6 error by those with controls which are
constant on uniform time intervals. In this note we refine the proof
and show that the provable rate can be improved to 1/4, which is
 optimal in our setting.
Moreover, we demonstrate the improvements this implies for error estimates derived by similar techniques for approximation schemes, bringing these in line with the
best available results from the PDE literature.
\end{abstract}


\section{Introduction}\label{sect:intro}

In this paper we derive improved error estimates for approximations of
value functions of stochastic optimal control problems. 
{ Let $(\Omega,\mathcal F, \{\mathcal F_t\}_{t\geq 0},\mathbb P)$ be a
  complete filtered probability space, $(W_t)_{t\ge 0}$ a
  $p$-dimensional $\{\mathcal F_t\}$-Wiener process on
  $(\Omega,\mathcal F,\mathbb P)$, and $\mathcal A$ the set of
  progressively measurable processes with values in a set
  $A\subseteq\R^m$. For any $\alpha\in\mathcal A$,  $x\in \R^d$, $t\in [0,T]$ (with $T>0$), let
$X_\cdot=X^{\alpha,t,x}_\cdot$ be the (controlled) It\^o diffusion which satisfies
\be\label{eq:SDE}
X_s = x + \int^{s}_0 b_{\alpha_r}(t+r,X_r) \, \mathrm{d} r +
\int^{s}_0 \sigma_{\alpha_r}(t+r,X_r) \, \mathrm{d}
W_r\qquad\text{for}\qquad s\geq t.
\ee
{Here we use the notation $\varphi_a(\cdot,\cdot) =
\varphi(\cdot,\cdot,a)$ for any $a\in A$ and function $\varphi$}.
For a given terminal cost function $g$ {and} running cost
$f$,
the optimal
control problem consists of maximizing over $\alpha\in \mathcal A$ the
expected total cost
\be
\label{def:J}
J^\alpha(t,x):=\E^{\alpha}_{t,x}\big[\int^{T-t}_0 f_{\alpha_r}(t+r,X_r)\,\mathrm{d} r + g(X_{T-t})  \big].
\ee
The indices on the expectation $\E$ indicate that the law of the process depends
on the starting point and control.  
Finally, the value function of the optimal control
problem is defined by
\be\label{eq:value_function}
v(t,x) := \underset{\alpha\in\mathcal A}{\sup}\, J^\alpha(t,x).
\ee
We consider the following set of assumptions:
\begin{itemize}
\item[\bf{(H1)}]
 $A$ is a compact set;
 \item[\bf{(H2)}] $b:[0,T]\times \R^d \times A\to\R^d$ and $\sigma:[0,T]\times \R^d\times A\to \R^{d\times p}$ are continuous functions. For
$\varphi\in \{b,\sigma\}$, there exists $C_0\geq 0$ such that  for every $t,s\in [0,T], x, y \in\R^d, a\in A$:
 \begin{eqnarray*}
 & &|\varphi(t,x,a)-\varphi(s,y,a)|\leq C_0  \left(|x-y| + |t-s|^{1/2}\right) \quad\text{and}\quad |\varphi(t,x,a)|\leq C_0;
 \end{eqnarray*}
\item[\bf{(H3)}] $g:\R^d \to\R$ and $f:[0,T]\times \R^d\times A\to \R$ are continuous functions. There exists $C_1\geq 0$ such that for every $t,s\in [0,T], x, y \in\R^d, a\in A$:
 \begin{eqnarray*}
 & &|g(x)-g(y)|\leq C_1 |x-y|, \\ 
 & &|f(t,x,a)-f(s,y,a)|\leq C_1  \left(|x-y| + |t-s|^{1/2}\right)\quad\text{and}\quad |f(t,x,a)|\leq C_1.
 \end{eqnarray*}
 
\end{itemize}
Observe that under assumptions (H1), (H2), and for any $\alpha\in \cA$,  there exists a unique strong solution of equation \eqref{eq:SDE}.  
For simplicity, we assume data  and coefficients to be Lipschitz continuous in space and $1/2$-H\"older continuous in time,
and have included no discount factor, but it is not difficult to extend
our results to include discounting and a lower  H\"older regularity for  $f$ and $g$.

We aim to estimate the error introduced by approximating the set of
measurable controls $\mathcal A$ by piecewise constant controls.
Let $h>0$
be the discretization parameter and $\mathcal A_\h$
the subset of $\mathcal A$ of processes which are constant in the
intervals $[n\h, (n+1)\h )$ for $n\in \N$.\footnote{Note that in  \cite{Kry99} the length of intervals is $h^2$, however, in absence of further discretisations, we use $h$ for simplicity.}
The value function associated with this restricted set of controls is defined by 
\be\label{eq:value_tau}
v_\h(t,x) := \underset{\alpha\in\mathcal A_\h}{\sup} J^\alpha(t,x).
\ee

Note that the definition of $\vh$ in (\ref{eq:value_tau}) under the ``shifted'' dynamics in (\ref{def:J}) and (\ref{eq:SDE}) implies that the control discretisation is
always centered at $t$. This will be important for establishing a dynamic programming principle. This is not, though, how one would compute $\vh$ in practice, as discussed in the penultimate
paragraph of this section.}

From a probabilistic perspective, it is clear that 0 is a lower bound for $v- v_h$ since $\mathcal A_\h\subseteq \mathcal A$. 
Under our assumptions, an upper bound on $v-v_h$ of order $h^{\frac16}$ is given in \cite{Kry99}.

An indication that the order 1/6 from \cite{Kry99} might be improved is the fact that under the same regularity assumptions as above it is shown in \cite{DJ12}
that a fully discrete semi-Lagrangian scheme applied to the corresponding HJB equation has order 1/4 in the timestep for an Euler approximation.
This scheme does not distinguish between constant or other controls over individual timesteps. It would therefore be somewhat surprising if the 
scheme which employs further approximations was closer to the original problem than the one which only holds the policies constant over timesteps.


A slightly different angle to the problem is provided in \cite{BJ07}, where the authors construct from (\ref{eq:value_tau}) a subsolution to the HJB equation
corresponding to (\ref{eq:value_function})
by a second order local expansion in $t$. This results in an order 1 error bound in the case of smooth solutions, in contrast to 1/2 which would be obtained in the smooth case by the method in \cite{Kry99} (see also Section \ref{subsec:regular} below).
However, in the general non-regular case, the order in \cite{BJ07} is limited by a switching system approximation of order $\vare^{1/3}$ (for a switching cost chosen of order $\vare$), which, combined with an error term
of the regularised system of order $\h/\vare^3$ (for regularisation parameter $\vare$), results in an order $1/10$ error by optimisation of $\vare$.

In this paper, we combine the advantages of both methods to obtain order 1/4.
The reason we can improve the error estimates of Krylov is that we
use a higher order expansion when we derive the truncation error. Our
discussion (see Subsection  \ref{subsec:regular}) also shows that no further improvement can be obtained in
this way: our new proof uses the maximal possible order of the
truncation error.

Piecewise constant policy time stepping has been used in a numerical method for solving Hamilton-Jacobi-Bellman equations in \cite{reisinger2016piecewise},
where the computational advantage comes from the fact that over the time intervals in which the policy is constant, only linear PDEs have to be solved.
This has been extended to mixed optimal stopping and control problems with nonlinear expectations and jumps in \cite{dumitrescu2018approximation}.
A further benefit lies in the inherent parallelism so that the linear problems with different controls can be solved on parallel processors.
A proof of convergence is given in these works using pure viscosity solution arguments, but no rate of convergence is provided.
Early results on this type of approximation can be found in \cite{lions1980formule} and an extension with ``predicted'' controls is proposed in \cite{kossaczky2016modifications}.

In the remainder of this article, we give in Section \ref{sec:main} a proof of the order 1/4 convergence of the piecewise constant policy approximation,
{and deduce the linear convergence in the case of sufficiently regular solutions and data.}
We then outline in Section \ref{sec:fd} the improved orders which can be derived for approximation schemes by similar techniques.

\section{Main result}
\label{sec:main}
{We begin by stating the main result. Throughout the entire section we work under assumptions
(H1)--(H3).
}

\begin{theo}
\label{theo:main}
For any $s \in [0,T]$, $x\in \R^d$, and $\h >0$, we have
\begin{eqnarray}
\label{eqn:main}
0 \ \leq\ v(s,x)-\vh(s,x) \ \leq \ C \h^{1/4},
\end{eqnarray}
{where the constant $C$ only depends on the constants in Assumptions (H2) and (H3).} 
\end{theo}

A major difficulty in the proof of Theorem \ref{theo:main} is the
fact that typically $v$ and $\vh$ are not smooth. Even in the
non-degenerate case where $v$ is $C^{2+\delta}$, $\vh$ is still not
smooth in general. A simple example is the Black-Scholes-Barenblatt
equation resulting from an uncertain volatility model {(see \cite{lyons1995uncertain})}. 
Here, the control is of bang-bang type and the optimal control problem for piecewise constant policies reduces to taking the maximum of two smooth functions at the end of each time interval, so that for $t$ on the time mesh,
$\vh(t,\cdot)$ will only be Lipschitz (in the spatial argument). 

Since the proof of Theorem \ref{theo:main} relies on repeated use of the It\^o formula, we need to
work with smooth functions, both for the coefficients and value functions
$v$ and $v_h$. This means that we need to introduce several
regularization arguments and use Krylov's  method of shaking the
coefficients. 


\subsection{Background results and regularisation}
\label{subsec:background}

{In this section, we introduce Krylov's regularization and give related preliminary
results.} Some
of the  proofs are given in \cite{Kry99} and not repeated here; 
see also \cite{BJ02,BJ05} for analogous results proved with PDE
arguments.
In order to apply It{\^o}'s formula twice,
$\sigma, b, f,g,v$, and $\vh$ must be regularized.  Let
$\vare>0$ and the mollifier $\rho_\vare$ be defined as 
\begin{align}
\label{eq:rhoxt}
\rho_\vare(t, x) :=
	\frac{1}{\vare^{d+2}}\rho \left( \frac{t}{\vare^2}, \frac{x}{\vare} \right),
\end{align}
 where
\[
\rho \in C^\infty(\R^{d+1}),\quad \rho\geq 0,\quad 
\mathrm{supp}\,\rho = (0,1) \times \{|x| < 1\}, \quad
\int_{\mathrm{supp}\,\rho}\rho(e) \, \mathrm{d} e =1.
\]
For any function $\varphi:[0,T]\times\R^d\to\R$, we define
$\varphi^{(\vare)}\in C^\infty([0,T]\times\R^d)$ to be the mollification of  a suitable extension of $\varphi$ to
$[-\vare^2, T]$ 
$$
\varphi^{(\vare)}(t,x):= (\varphi*\rho_\vare)(t,x) =  \int_{0\leq s\leq \vare^2}\int_{|y|\leq \vare} \varphi(t-s,x-y)\rho_\vare(s,y) \, \mathrm{d}s \, \mathrm{d}y.
$$
{We can always take an extension which preserves
the H\"older continuity in time and Lipschitz continuity in space of $\varphi$. 
Then standard estimates for mollifiers imply that}
\be\label{eq:deriv_coeff}
\|\varphi-\varphi^{(\vare)}\|_\infty\leq C\vare\qquad\text{and}\qquad\left\|\partial^m_t D^{k}_x \varphi^{(\vare)}\right\|_\infty \leq C {\vare^{1-2m-k}}\ \qquad \text{for}\qquad {k+m}\geq 1.
\ee

Let $\tilde X_\cdot$ be the solution of \eqref{eq:SDE} with coefficients
replaced by $b^{(\vare)}$ and $\sigma^{(\vare)}$. Then we denote by
$\tilde v$ and $\tilde J^\alpha$ the solution and cost function of
the optimal control problem \eqref{eq:SDE}--\eqref{eq:value_function}
where $X_\cdot$ 
is replaced by $\tilde X_\cdot$ and $f,g$ by $f^{(\vare)},g^{(\vare)}$.

\begin{prop}\label{tildev}
There exists $C\geq 0$ such that for any $t\in [0,T], x\in\R^d$
$$
|v(t,x)-\tilde v(t,x)|\leq C\vare.
$$
\end{prop}
\begin{proof}
The result follows from the definitions of $v$ and $\tilde
v$ since by standard  continuous
dependence results for SDEs and Lipschitz and H\"older continuity of
$f,g, b, \sigma$,
$$
\E^{\alpha}_{t,x}\Big[\underset{s\in [0,T-s]}\sup |X_s - \tilde
  X_s|^2\Big]\leq C (\|b-b^{(\vare)}\|^2_\infty +
\|\sigma-\sigma^{(\vare)}\|^2_\infty)\leq C\vare^2
$$
for some constant $C$ independent of the control $\alpha$.
\end{proof}
To avoid heavy notation, we will use $(f,g,b,\sigma)$ instead of $(f^{(\vare)},g^{(\vare)},b^{(\vare)},\sigma^{(\vare)})$ in the rest of the paper,
keeping in mind estimates \eqref{eq:deriv_coeff} for their derivatives.  
We now proceed with the regularisation of the value function $v_\h$.
Let $\mathcal E_\h$
be the set of progressively measurable processes 
$e\equiv(e_1,e_2)$ with values in  $(-\vare^2,0)\times B_\vare(0)$ (where $B_\vare(0)$ denotes the ball of radius $\vare$ in $\R^d$) which are constant in each time interval $[n\h,(n+1)\h )$.
Letting $S = T+\vare^2$, we define for any $s\in [0,S], x\in \R^d$  the following ``perturbed'' value function 
\be\label{eq:veps}
\uh(s,x):=\underset{\alpha\in\mathcal A_{\h}, e\in\mathcal E_\h}\sup
\mathbb E^{(\alpha, e)}_{s,x} \Big[ \int^{S-s}_0
  f_{\alpha_r}(s+r,\hat X_r) \,\mathrm{d} r + g(\hat X_{S-s}) \Big],
\ee
where $\hat X_\cdot = \hat X^{(\alpha,e),s,x}_\cdot$ is the solution of the
following SDE with (mollified and) ``shaken coefficients'':
\begin{align}\label{SDE shake}
\hat X_\cdot = x +\int^{\cdot}_0 b_{\alpha_r}(s + r + e_{1,r},\hat  X_r + e_{2,r}) \,\mathrm{d}r + \int^{\cdot}_0 \sigma_{\alpha_r}(s + r + e_{1,r},\hat  X_r + e_{2,r}) \, \mathrm{d}W_r.
\end{align}

\begin{prop}\label{prop:vhe}
There exists  a constant $C\geq 0$ such that
$$
|\vh (t,x)-\uh(t,x)|\leq C \vare
$$
for any $t\in [0,T], x\in \R^d$, and
$$
|\uh (t,x)- \uh(s,y)|\leq C (|x-y| + |t-s|^{1/2})
$$
for any $t,s\in [0,S]$ and $x,y\in\R^d$. Moreover, for any $s\in [0,S-\h]$,  $\uh$ satisfies the following dynamic programming principle (DPP):
\be\label{eq:dpp_vhe}
\uh(s,x)=\underset{\substack{a\in A\\ 0\leq \eta\leq \vare^2, |\xi|\leq \vare}}\sup \mathbb E^{(a,(\eta,\xi))}_{s,x} \Big[ \int^{\h}_0 \fa(s+r,\hat X_r)\, \mathrm{d} r + \uh(s+\h,\hat X_{\h})  \Big].
\ee
\end{prop}
\begin{proof}
These are standard results. The first two inequalities can be found
e.g. in \cite[Corollary 3.2]{Kry99}, while \eqref{eq:dpp_vhe} is a consequence of \cite[Lemma 3.3]{Kry99}.
\end{proof}

Following the notation introduced above 
we consider the regularised (mollified) function $\uhe$.

\begin{prop}
\label{pro:supdpp}
The function $\uhe$  belongs to $C^\infty([0,T]\times\R^d)$. There exists a constant $C\geq 0$ such that 
\be\label{eq:uh-uhe}
\big|\uh(t,x)-\uhe(t,x)\big|\leq C\vare
\ee
for $t\in [0,T],x\in\R^d$, and  
\be\label{eq:deriv}
\left\|\partial^m_t D^{k}_x \uhe\right\|_\infty \leq C \vare^{1-2m}\vare^{1-k}
\ee
for any $k, m\geq 1$.
Moreover, $\uhe$ satisfies the following super-dynamic programming principle 
\begin{equation}\label{eq:DPP_uhe}
\uhe(t,x)\geq  E^{a}_{t,x} \Big[\int^{\h}_0 \fa(t+r,\tilde X_r)
  \,\mathrm{d} r + \uhe(t+\h,\tilde X_{\h})  \Big]
\end{equation}
for any $a\in A$, $0\leq \eta\leq \vare^2, |\xi|\leq \vare$, $t\in [0,T-h], x\in\R^d$.
\end{prop}


\begin{proof}
The first part follows
from Proposition \ref{prop:vhe} and
\eqref{eq:deriv_coeff},
while \eqref{eq:DPP_uhe} follows by the definitions of $\uhe$, $\hat
X_t$, $\tilde X_t$, and
the inequality $\int \sup (\cdots)\geq \sup\int(\cdots)$. See
\cite[bottom of page 9]{Kry99} for more details. Here $\alpha_t \equiv a$ constant over 
$t\in [0,h]$ by a slight abuse of notation. 
\end{proof}

\subsection{Proof of  Theorem \ref{theo:main}}\
\label{subsec:proof}

\medskip
\noindent 1)\quad {\em Upper bound on $\Lae \uhe+\fae$.}
By  two applications of the {It\^o} (or Dynkin) formula,
\begin{align*}
\Easx[\uhe(s+\h,\tilde X_h)]
&=  \uhe(s,x)+ \Easx \Big[ \int_0^h (\Lae \uhe)(s+t,\tilde X_t) \dt\Big]\\
&= \uhe(s,x)+\h (\Lae \uhe)(s,x) + \Easx \Big[ \twoint \Lae (\Lae
  \uhe)(s+r,\tilde X_r) \dr \dt\Big]
\end{align*}
for $s \le T-\h$, $x\in \mathbb{R}^{{d}}$, $a\in A$, where
the generator $\La$ of the diffusion process is defined as
\[
\La := \partial_t + b_a^T D_x + \frac{1}{2}tr[\sigma_a \sigma_a^T D_x^2].
\]
Inserting this equality into the dynamic programming inequality
\eqref{eq:DPP_uhe} in Proposition \ref{pro:supdpp}, applying It{\^o}
once to the $\fae$-term, and  
dividing by $\h$, we find that
\begin{align}
 (\Lae \uhe)(s,x) + \fae(s,x)
\le 
\frac1\h \sup_{a\in A} \Big(\|\Lae \fae\|_\infty + \|\Lae \Lae
\uhe\|_\infty\Big) \twoint \dr \dt.
\label{Linequ}
\end{align}
Since the leading term $\Lae \Lae \uhe$ is a sum of terms of the form $\phi_1 (\partial_t^m \phi_2) (D_x^k \phi_3)$ with
$\phi_i \in \{\mu, \sigma \sigma^T, \uhe\}$ and $2m+k\le 4$, by
(\ref{eq:deriv_coeff}) and (\ref{eq:deriv}),
\begin{align}
(\Lae \uhe)(s,x)+\fae(s,x) \leq C \vare^{-3}\h.
\label{Lainequ}
\end{align}
\smallskip

\noindent 2)\quad {\em Upper bound on $\tilde v-\vh$ for $s\in[0,T-h)$.}
Let $\alpha \in \mathcal A$, $s\in [0,T-\h]$, and $x\in\R^d$. By
It{\^o}'s formula and part 1), 
\begin{align*}
\Ealsx[ \uhe(T-h,\tilde X_{T-h-s})] & \le \uhe(s,x)+\Ealsx\Big[ \int_0^{T-s}(\Lat
  \uhe)(s+t,\tilde X_t) \dt \Big]\\
&\le \uhe(s,x)- \Ealsx\Big[ \int_0^{T-s} \fat(s+t,\tilde X_t) \dt \Big]+ TC
\vare^{-3} {\h}.
\end{align*}
From \eqref{eq:uh-uhe} in Proposition \ref{pro:supdpp} and the
first part of Proposition \ref{prop:vhe}, it then follows that
\begin{align*}
\Ealsx[ \uh(T-h,\tilde X_{T-h-s})]  &\le \uh(s,x)- \Ealsx\Big[ \int_0^{T-s}
  \fat(s+t,\tilde X_t) \dt \Big] + C (\vare + \vare^{-3} {\h})\\
&\le \vh(s,x)- \Ealsx\Big[ \int_0^{T-s} \fat(s+t,\tilde X_t) \dt \Big] + C (\vare + \vare^{-3} {\h}),
\end{align*}
for a generic constant $C$. 
Since by definition \eqref{eq:veps} and {the} regularity of $\uh$
(Proposition \ref{prop:vhe}), 
\begin{align*}
\Ealsx [(\uh(T-h,\tilde X_{T-h-s})] & =  
\Ealsx[ \uh(T-h,\tilde X_{T-h-s}) - \uh(S,\tilde X_{T-s})+g(\tilde X_{T-s}))]\\
& \ge \Ealsx [g(\tilde X_{T-s}))]- C (\h^{1/2} + \vare),
\end{align*}
we conclude that 
\begin{eqnarray*}
\tilde J^\alpha(s,x)=\Ealsx \Big[\int_0^{T-s} \fat(s+t,\tilde X_t) \dt + g(\tilde X_{T-s}) \Big] \le \vh(s,x) + C (\vare + \h^{1/2} + \vare^{-3} {\h}).
\end{eqnarray*}
Since $\alpha\in\mathcal A$ was arbitrary, by the definition
 of $\tilde v$ (see just before Proposition \ref{tildev}),
\[
\tilde v(s,x) - \vh(s,x) \le C (\vare + \h^{1/2} + \vare^{-3} {\h}).
\]
\smallskip

\noindent 3)\quad {\em Upper bound on $\tilde v-\vh$ for $s\in[T-h,T]$.}
  By the definition of $\tilde J^\alpha$ (see just before Proposition \ref{tildev}), It\^o's
  formula, the regularity of $f$ and
$g$, and using \eqref{eq:deriv_coeff}, there is a constant $C>0$ such that for every
$\alpha\in\mathcal A$ and $s\in[T-h,T]$,
$$
|\tilde J^\alpha(s,x) -g(x)| = 
\Big| \Ealsx\Big[ \int^{T-s}_0
   \Big(f_{\alpha_r}(s+r,\tilde X_r)+L_{\alpha_r} g(\tilde X_{r}) \Big)\,\mathrm{d}
   r\Big]\Big|\leq C(1+\vare^{-1})h.
$$
Then it follows from the definitions of $\tilde v$ and $\vh$ that
$$|\tilde v(s,x)-g(x)|+|\vh(s,x)-g(x)|\leq C\vare^{-1}h,
$$
and hence also $|\tilde v(s,x)-\vh(s,x)|\leq 2C\vare^{-1}h$ for $s\in[T-h,T]$.

\medskip
\noindent 4)\quad {\em Conclusion:} Using Proposition
\ref{tildev} and parts 2) and 3), we have that 
$$v(s,x)-\vh(s,x)\leq \tilde v(s,x)-\vh(s,x)+ C\vare \leq  C (\vare + \h^{1/2} + \vare^{-3} {\h})$$ for
$s\in[s,T]$ and $x\in\R^d$. Taking $\vare=\h^{1/4}$ then concludes the
proof of the right-hand inequality in \eqref{eqn:main}. The left-hand inequality is immediate since 
$\mathcal{A}_h \subseteq \mathcal{A}$.

\subsection{The maximal rate and comparison with \cite{Kry99}}
\label{subsec:regular}

If the data and value functions are smooth enough,
we {can adapt the proof of Theorem \ref{theo:main}} to obtain the maximal rate of the approximation, which is 1.
{More specifically,} {if we assume $v_h$ {and $f$} sufficiently smooth, we have  {in (\ref{Linequ})} $\sup_{a\in A}(\|\La(\La\uhe)\|_\infty+\|\La
f\|_\infty) {\le C} <\infty$ {with $C$ independent} of $\vare$. Therefore, instead of (\ref{Lainequ})}, the conclusion of step 1) in the previous proof gives
$$(\Lae \uhe)(s,x)+\fae(s,x) \leq C \h,$$
for some constant $C$ independent of $a\in A$ and $\vare$.
Moreover, if we assume that $b$, $\sigma$ and $f$ are
Lipschitz in $t$ uniformly in $x$ and $a$, and $g$ belongs to $C_b^2 (\R^d)$, then by standard
results $u_h$ will be Lipschitz in $t$. Hence, we find in step 2) that
\[
\tilde v(s,x) - \vh(s,x) \le C (\vare +  {\h}).
\]
Sending $\vare$ to zero then gives that $\tilde v$ converges to $v$, and
we have the following result:

\begin{prop}
Additionally to assumptions (H1)-(H3), let $b, \sigma$ and $f$ be Lipschitz continuous in $t$ uniformly with respect to $x$ and $a$, and $g\in C^2_b(\R^n)$. If $\sup_{a\in A}(\|\La(\La\vh)\|_\infty+\|\La
f\|_\infty) <\infty$, then there exists $C>0$ such that for any
  $s \in [0,T]$, $x\in \R^d$, and $\h >0$, we have 
\begin{eqnarray}
\label{eqn:main2}
0 \ \leq\ v(s,x)-\vh(s,x) \ \leq \ C \h.
\end{eqnarray}
  \end{prop}

This is the maximal rate that this approximation can reach. The reason
is that the order obtained by applying It\^o twice in step 1)
of the proof cannot be improved. This can easily be checked by repeatedly
applying It\^o  to obtain higher order error expansions and then noting that all such
expansions contain terms of order $h$.

Step 1) of the proof also explains why Krylov in \cite{Kry99} got a
 less sharp result than ours. After one application of It\^o, he used the 
moment bound $\E[|x-X_r|]\leq \sqrt{\E[|x-X_r|^2]}\leq C\sqrt r$ to
get 
\begin{align*}
\Big|\frac1h\Easx \Big[ \int_0^h (\Lae \uhe)(s+t,\tilde X_t) \dt\Big]
- (\Lae \uhe)(s,x)\Big| \leq {C} \|D_x (\Lae \uhe)\|_\infty h^{1/2}  {+ \|  \partial_t  (\Lae \uhe)\|_\infty h}.
\end{align*}
This estimate requires only {three derivatives in space} of $\uhe$ but gives the
lower rate 1/2. The conclusion of step 1) of the proof then becomes
\begin{eqnarray*}
\La \uhe(s,x) + \fa(s,x) \le C \left(\vare^{-2} h^{1/2}  + \vare^{-3} h  \right).
\end{eqnarray*}
Completing the proof as in Section \ref{subsec:proof} then gives
$$\tilde v(s,x) - \vh(s,x) \le C (\vare + \vare^{-2} h^{1/2} + \vare^{-3} h),$$
and optimizing with respect to $\vare$ shows that $ v(s,x) - \vh(s,x)
\le Ch^{1/6}$. Note that there is
no need for regularization of the coefficients and data since It\^o is
applied only once. In the case of smooth enough solutions, this approach
cannot give a higher rate than $1/2$. 


\section{Consequences on finite difference approximations}
\label{sec:fd}

In this section, we outline the impact of the improved error bound for the control approximation on the achievable convergence order for numerical schemes,
either by directly substituting the improved order {(Section \ref{subsec:impr1})} or by applying adaptations of the steps here using higher order estimates {(Section \ref{subsec:impr2})}.

\subsection{Improvement to Theorem 1.11 in \cite{krylov2000rate}}\label{subsec:impr1}


Using the new bound for the control approximation from Section \ref{sec:main}, one easily obtains a sharpening of the order from $1/39$ in \cite[Theorem 1.11]{krylov2000rate} and $1/21$ in \cite[Theorem 5.4]{Kry99} to $1/15$,
which holds for local, monotone schemes of consistency order $1/2$.
Indeed, using Theorem \ref{theo:main} instead of \cite[Theorem 2.3]{Kry99}, the bound in the second inequality in the proof of \cite[Theorem 5.4]{Kry99}
(on top of page 14 in \cite{Kry99})
becomes 
$$
v\leq  v_{\delta,1/n} + C(n \delta^{1/3} + n^{-1/4}),
$$
where $\delta>0$ is the {time} discretization step used in \cite{Kry99} for the approximation {scheme for} the value function,  $n$ the number of
{time intervals over which the policy is constant} and $v_{\delta,1/n}$ is the obtained approximation of $v$.\footnote{
Note that in Section 5 of \cite{Kry99}, our $\delta$ above is denoted by $h^2$. We introduce $\delta$ to avoid ambiguity with the parameter $h$ used in the previous sections of this paper (corresponding to $h = 1/n$ in the present section).
}
Optimizing with respect to $\delta$ gives $n \sim \delta^{-4/15}$ and
 an estimate of 
 order $1/15$ in $\delta$.

Assuming order 1 consistency {of the scheme used} instead of order 1/2 as in \cite[Theorem 1.11]{krylov2000rate} and \cite[Theorem 5.4]{Kry99}, in conjunction with \cite[Lemma 3.2]{krylov2000rate},
one gets 
$$
v\leq  v_{\delta,1/n} + C(n \delta^{1/2} + n^{-1/4}),
$$
and the rate improves further to $1/10$.

\subsection{Improvement to Theorem 5.7 in \cite{Kry99}}\label{subsec:impr2}

For a wide class of numerical schemes, similar modifications as those used to prove Theorem \ref{theo:main} can be performed to improve the error estimates  given in  \cite[Theorem 5.7]{Kry99}. 
Following as much as possible the notation in \cite{Kry99}, let us define for any $s\geq 0$, $x\in\R^d$, $a\in A$ the random variable
$$
Y^{a,s,x}:= x + b(s,x,a) h + \sigma(s,x,a)  \zeta,
$$
where $\zeta$ is an $\R^p$-valued random variable  such that 
\be\label{eq:moments}
\E[\zeta] = 0, \quad 
|\E[\zeta_i \zeta_j] - h \delta_{ij}| \le C h^2
\quad\text{and}\quad \E[\zeta^k] \leq C h^2 \text{ for any $k\geq 3$}.
\ee
It is easy to check, by Taylor expansion, that for any smooth function $\phi$ the estimate in \cite[Lemma 5.10]{Kry99} for the truncation error of the generator becomes
$$
\left| \La  \phi (s,x) - h^{-1} \E \left[ \phi(s+h,Y^{a,s,x}) - \phi(s,x)\right]\right| \leq C h
$$
for a constant $C$ depending only on $C_1$ and $C_2$ in assumptions (H2)--(H3) and the bounds on the derivatives $\partial^m_t D^k_x \phi$ for $2m +k\leq 4$.

Observe that conditions \eqref{eq:moments} are slightly stronger than (5.4) in \cite{Kry99}, who only assume accuracy of the moments to order $h^{{3/2}}$ instead of $h^2$
in  \eqref{eq:moments}, so that only order $1/2$ consistency results instead of order 1 above.
However, the higher order assumptions are satisfied by very common schemes such as the classical semi-Lagrangian scheme \cite{CF95, DJ12} corresponding to the choice
\be\label{eq:sl}
\mathbb P(\zeta_i = \pm h^{1/2} ) = 1/2\qquad \text{for $i=1,\ldots,p$}.
\ee

The scheme considered in \cite{Kry99} is then recursively defined, for any $x\in \R^d$, by 
\begin{align*}
& \hat v_h (s,x) = g(x)  & & \text{if}\quad s\in (T-h,T], \\
& \hat v_h (s,x) = \underset{a\in A}\sup\; \left\{ \fa(s,x)h + \E\left[ \hat v_h(s+h, Y^{a,s,x})\right] \right\} & & \text{if}\quad s\leq T-h. 
\end{align*}

Proceeding to a perturbation and regularization of $\hat v_h$ as in \cite{Kry99} (the notation follows the one in Section \ref{subsec:proof}, i.e. $\hat u_h^{(\vare)}$ is the mollification of $\hat u_h$, the solution of the scheme with perturbed ``shaken'' coefficients)  we get the inequality 
$$
\La \hat u^{(\vare)}_h  + \fa \leq C h \vare^{-3} 
$$
in $[0,T-h]\times \R^d$ for some constant $C$ depending only on $C_0, C_1$ in assumptions (H2) and (H3). Arguing as in  the proof of Theorem \ref{theo:main}, one obtains 
$$
\hat v_h\leq v + C h^{1/4}.
$$
Similarly, an upper bound of order $1/4$ for $v-\hat v_h$ can be obtained. This aligns the bounds for the scheme \eqref{eq:sl} with those obtained in \cite{DJ12} by PDE techniques. 

\section{Discussion and conclusions}

In this {short} paper, we show a convergence rate of $1/4$ for piecewise
constant control approximations to value functions of 
stochastic optimal control problems. This result is robust and holds for
degenerate problems with non-smooth, merely Lipschitz continuous value
functions. If the {data} and value function are smoother, we
show that the approximation has rate 1 and explain
why this is the maximal rate.  

Our rate 1/4 in \eqref{eqn:main} improves both the order 1/6 in
\cite{Kry99} and the rate 1/10 achieved in \cite{BJ07} by different (PDE)
techniques. We also carefully explain why we can improve the
result in \cite{Kry99}. {It is an interesting open question if the same rate could be obtained purely by
 PDE techniques.}




This work also opens up the possibility of improving the error
estimates for other approximation schemes as outlined in Section \ref{sec:fd}.
Moreover, it enables a purely probabilistic error analysis for
semi-Lagrangian schemes for HJB equations with results that are in line
with the best available results by PDE methods. We refer to \cite{PicaReisiErrors18} for the
details.

\bibliography{biblio.bib}
\bibliographystyle{plain}

\end{document}